\documentclass[leqno,10pt]{amsart}

\usepackage[latin1]{inputenc}
\usepackage[T1]{fontenc}
\usepackage[english]{babel}
\usepackage{amsmath}
\usepackage{verbatim}
\usepackage{amsfonts}
\usepackage[all]{xy}
\usepackage{graphicx}
\usepackage{amssymb}
\usepackage{MnSymbol}
\usepackage{amsthm}
\usepackage{hyperref}

\newcommand{\TT}{\mathcal{T}}
\newcommand{\YY}{\mathcal{Y}}
\newcommand{\FF}{\mathcal{F}}
\newcommand{\OO}{\mathcal{O}}
\newcommand{\HH}{\mathcal{H}}
\newcommand{\HHp}{\mathcal{H}^{\mathrm{main}}}
\newcommand{\HHx}{\mathcal{H}_I}
\newcommand{\HHy}{\mathcal{H}_{I\!I}}
\newcommand{\HHxp}{\mathcal{H}_{I}^{\mathrm{main}}}
\newcommand{\HHyp}{\mathcal{H}_{I\!I}^{\mathrm{main}}}
\newcommand{\BB}{\mathcal{B}}
\newcommand{\dd}{\mathbf{d}}
\newcommand{\NN}{\mathbb{N}}
\newcommand{\Gm}{\mathbb{G}_m}

\newcommand{\PP}{\mathbb{P}}
\newcommand{\CC}{\mathbb{C}}
\renewcommand{\gg}{\mathfrak{g}} 
\newcommand{\hh}{\mathfrak{h}} 
\newcommand{\gl}{\mathfrak{gl}} 
\renewcommand{\sl}{\mathfrak{sl}} 
\renewcommand{\sp}{\mathfrak{sp}} 
\newcommand{\uu}{\mathfrak{u}} 
\newcommand{\so}{\mathfrak{so}} 
 
\newcommand{\Irr}{\mathrm{Irr}}

\newcommand{\Gr}{\mathrm{Gr}}
\newcommand{\IG}{\mathrm{IGr}}
\newcommand{\OG}{\mathrm{OGr}}
\newcommand{\Hom}{\mathrm{Hom}}
\newcommand{\Ad}{\mathrm{Ad}}
\newcommand{\Mor}{\mathrm{Mor}}
\newcommand{\tr}{\mathrm{tr}}
\newcommand{\Ker}{\mathrm{Ker}}
\renewcommand{\Im}{\mathrm{Im}}
\newcommand{\Aut}{\mathrm{Aut}}
\newcommand{\Stab}{\mathrm{Stab}} 
\newcommand{\End}{\mathrm{End}}
\newcommand{\Hilb}{\mathrm{Hilb}}
\newcommand{\rk}{\mathrm{rk}}
\newcommand{\Spec}{\mathrm{Spec}}

\renewcommand{\labelitemi}{$\bullet$}

\newcommand{\leftexp}[2]{{\vphantom{#2}}^{#1}{#2}}

\theoremstyle{plain}
\newtheorem{theoreme}{Theorem}[section]
\newtheorem{lemme}[theoreme]{Lemma}
\newtheorem{proposition}[theoreme]{Proposition}
\newtheorem*{lemme*}{Key-Proposition}
\newtheorem*{proposition*}{Reduction Principle}
\newtheorem*{proposition**}{Proposition C}
\newtheorem*{theoreme**}{Theorem A}
\newtheorem*{theoreme***}{Theorem B}
\newtheorem{corollaire}[theoreme]{Corollary}

\theoremstyle{definition}

\newtheorem{conjecture}[theoreme]{Conjecture}

\theoremstyle{remark}
\newtheorem{remarque}[theoreme]{Remark}

\begin{document}

\title[Invariant Hilbert schemes and symplectic reductions]{Invariant Hilbert schemes and desingularizations of symplectic reductions for classical groups}

\author{Ronan TERPEREAU}

\begin{abstract}
Let $G \subset GL(V)$ be a reductive algebraic subgroup acting on the symplectic vector space $W=(V \oplus V^*)^{\oplus m}$, and let $\mu:\ W \rightarrow Lie(G)^*$ be the corresponding moment map. In this article, we use the theory of invariant Hilbert schemes to construct a canonical desingularization of the symplectic reduction $\mu^{-1}(0)/\!/G$ for classes of examples where $G=GL(V)$, $O(V)$, or $Sp(V)$. For these classes of examples, $\mu^{-1}(0)/\!/G$ is isomorphic to the closure of a nilpotent orbit in a simple Lie algebra, and we compare the Hilbert-Chow morphism with the (well-known) symplectic desingularizations of $\mu^{-1}(0)/\!/G$.
\end{abstract}

\maketitle

\setcounter{tocdepth}{1}

\tableofcontents

\section{Introduction and statement of the main results}  \label{introo}

First of all, let us recall briefly the definition of the \textit{invariant Hilbert scheme}, constructed by Alexeev and Brion (see \cite{AB,Br} for more details). We work over the field of complex numbers $\CC$. Let $G$ be a reductive algebraic group, and let $h:\ \Irr(G) \rightarrow \NN$ be a \textit{Hilbert function} which assigns to every irreducible representation of $G$ a nonnegative integer. If $X$ is an (possibly reducible) affine $G$-variety, then the invariant Hilbert scheme $\Hilb_{h}^{G}(X)$ is the moduli space that parametrizes the $G$-stable closed subschemes $Z$ of $X$ such that
\begin{equation*}
\CC[Z] \cong \bigoplus_{M \in \Irr(G)} M^{\oplus h(M)}
\end{equation*}
as a $G$-module. Let us now suppose that the categorical quotient 
$$X/\!/G:=\Spec(\CC[X]^G)$$ 
is an irreducible variety. If $h=h_0$ is the Hilbert function of the \textit{general fibers} of the quotient morphism $\nu:\ X \rightarrow X/\!/G$ (that is, the fibers over a nonempty open subset of $X/\!/G$), then there exists a projective morphism 
\begin{equation*}
\gamma:\ \Hilb_{h_0}^{G}(X) \rightarrow X/\!/G,
\end{equation*}
called the \textit{Hilbert-Chow morphism}, that sends a closed subscheme $Z \subset X$ to the point $Z/\!/G \subset X/\!/G$. 
The Hilbert-Chow morphism induces an isomorphism over the flat locus $U \subset X/\!/G$ of $\nu$. The \textit{main component} of $\Hilb_{h_0}^G(X)$ is the irreducible component defined by 
\begin{equation*}
\Hilb_{h_0}^G(X)^{\mathrm{main}}:=\overline{\gamma^{-1}(U)}.
\end{equation*}
Then the restriction $\gamma:\ \Hilb_{h_0}^G(X)^{\mathrm{main}} \rightarrow X/\!/G$ is a projective birational morphism, and thus $\gamma$ is a candidate for a canonical desingularization of $X/\!/G$. It is an open problem to determine whether this restriction is always a desingularization or not. Last, but not least, if $H$ is any algebraic subgroup of the $G$-equivariant automorphism group $\Aut^G(X)$, then $H$ acts on $X/\!/G$ and $\Hilb_{h_0}^{G}(X)$, and the quotient morphism $\nu:\ X \rightarrow X/\!/G$ and the Hilbert-Chow morphism $\gamma:\ \Hilb_{h_0}^{G}(X) \rightarrow X/\!/G$ are $H$-equivariant.

Let now $G$ be an algebraic group, let $\gg$ be the Lie algebra of $G$, and let $W$ be a \textit{symplectic $G$-module}, that is, a $G$-module equipped with a $G$-invariant non-degenerate skew-symmetric bilinear form. Then $W$ has a $G$-equivariant \textit{moment map} 
$$\mu_G:\ W  \rightarrow  \gg^*,$$ 
which is defined in the usual way. To simplify the notation, we will use $\mu$ instead of $\mu_G$. The map $\mu$ being $G$-equivariant, the \textrm{set-theoretic fiber} $\mu^{-1}(0)$ is a $G$-stable subvariety of $W$. From now on, we suppose that $G$ is reductive. The categorical quotient $\mu^{-1}(0)/\!/G$ is called the \textit{symplectic reduction} of $W$ by $G$ and plays a central role in the study of $W$. It is an important problem to determine whether $\mu^{-1}(0)/\!/G$ admits a \textit{symplectic desingularization} (which is a distinguished desingularization, see Section \ref{gen} for details); a candidate for such a desingularization is given by the Hilbert-Chow morphism $\gamma:\ \Hilb_{h_0}^{G}(\mu^{-1}(0))^{\mathrm{main}} \rightarrow \mu^{-1}(0)/\!/G$.

Let us take $V$ a finite dimensional vector space, and $m$ a nonnegative integer. In this paper, we are interested in the cases where \begin{equation*}  
W:=(V \oplus V^*)^{\oplus m},
\end{equation*} 
on which $GL(V)$ acts naturally, and $G=GL(V)$, $O(V)$, or $Sp(V)$. In this situation, we can find a classical algebraic subgroup $H \subset \Aut^G(\mu^{-1}(0))$, namely
\begin{enumerate}
\item $H=GL_{m}$ for $G=GL(V)$;
\item $H=Sp_{2m}$ for $G=O(V)$; and
\item $H=SO_{2m}$ for $G=Sp(V)$.
\end{enumerate}
In each case, we will see that the symplectic reduction $\mu^{-1}(0)/\!/G$ identifies with a nilpotent orbit closure in the Lie algebra $\hh$ of $H$, except in Case (3) for $m$ even and $m\leq \dim(V)$ where it is the union of two such orbit closures. In Case (3), if $\mu^{-1}(0)/\!/G=Y_1 \cup Y_2$ is reducible, then \textbf{we always consider only one component} to simplify the statements (that is, $\mu^{-1}(0)/\!/G$ should be replaced by $Y_i$ and $\mu^{-1}(0)$ by $\nu^{-1}(Y_i)$). The geometry of nilpotent orbits has been extensively studied by Fu, Kraft, Namikawa, Procesi...(\cite{KP2,KP4,KP3,FuB,Fu2,Fu3,Nam}). In particular, the normalizations of such closures are \textit{symplectic varieties} (as defined by Beauville in \cite{Beau}) whose symplectic desingularizations are the so-called \textit{Springer desingularizations}, obtained by collapsing the cotangent bundle over some flag varieties (see Section \ref{gen} for details).

In \cite{Terp,Terp1}, we studied the invariant Hilbert scheme for classical groups acting on classical representations. We obtained classes of examples where the Hilbert-Chow morphism is a desingularization of the categorical quotient, and further examples where it is not. In this article, we use the results of \cite{Terp1} to prove the following statements:

\begin{theoreme**} \emph{([Sections \ref{proofGLn} and \ref{pprrthe}])} 
With the above notation, let $G=GL(V)$, $O(V)$, or $Sp(V)$, then the Hilbert-Chow morphism $\gamma:\ \Hilb_{h_0}^{G}(\mu^{-1}(0))^{\mathrm{main}} \rightarrow \mu^{-1}(0)/\!/G$ is a symplectic desingularization (and the unique one) if and only if 
\begin{itemize}
 \item $G=GL(V)$, $\dim(V) \geq m-1$, and $m$ is even; or
 \item $G=O(V)$, and $\dim(V) \geq 2m-1$; or
 \item $G=Sp(V)$, $\dim(V)$ and $m$ are even, and $\dim(V) \geq 2m-2$.
\end{itemize}
\end{theoreme**}

\begin{theoreme***} \emph{([Sections \ref{proofGLn} and \ref{pprrthe}])} 
With the above notation, let $G=GL(V)$, $O(V)$, or $Sp(V)$, then the Hilbert-Chow morphism $\gamma:\ \Hilb_{h_0}^{G}(\mu^{-1}(0))^{\mathrm{main}} \rightarrow \mu^{-1}(0)/\!/G$ is a desingularization that strictly dominates the symplectic desingularizations (when they exist) in the following cases:  
\begin{itemize}
 \item $G=GL(V)$ and either $\dim(V)=1$, $m \geq 3$ or $\dim(V)=2$, $m \geq 4$; or
 \item $G=O(V)$ and either $\dim(V)=1<m$ or $\dim(V)=2 \leq m$; or
 \item $G=Sp(V)$ and either $\dim(V)=2<m$ or $\dim(V)=4 \leq m$.
\end{itemize}
\end{theoreme***}

If $G \subset GL(V)$ is any reductive algebraic subgroup, then it is generally a difficult problem to determine whether $\Hilb_{h_0}^{G}(\mu^{-1}(0))$ is irreducible, that is, equals its main component. In this direction, we obtain   

\begin{proposition**} \emph{([Propositions \ref{casSympn1} and \ref{wxcv}])} 
With the above notation, if $G=GL(V)$ and $m \geq 2\dim(V)$, then the invariant Hilbert scheme $\Hilb_{h_0}^{G}(\mu^{-1}(0))$ has at least two irreducible components (and exactly two when $\dim(V)=1$). On the other hand, if $G=O(V)$ or $Sp(V)$, and $m \geq \dim(V)=2$, then $\Hilb_{h_0}^{G}(\mu^{-1}(0))$ is irreducible. 
\end{proposition**}

In Section \ref{gen}, we recall some basic facts about symplectic varieties and closures of nilpotent orbits in simple Lie algebras. The case of $GL(V)$ is treated in Section \ref{posisimpy}, and the case of $Sp(V)$ is treated in Section \ref{posisimpy2}. The case of $O(V)$ is quite similar to the case of $GL(V)$, and details can be found in the thesis \cite[\S 3.4]{Terp} from which this article is extracted. Besides, we think that our methods also apply when $G=SL(V)$, while the case $G=SO(V)$ should be more involved.\\

\section{Generalities on symplectic varieties and closures of nilpotent orbits}  \label{gen}

\subsection{Symplectic varieties and symplectic desingularizations}

Let us first recall the definitions of symplectic variety and symplectic desingularization (see \cite{Beau} or the survey \cite{FuBB} for more details).  Let $Y$ be a normal variety whose regular locus $Y_{\mathrm{reg}}$ admits a \textit{symplectic form} $\Omega$ (that is, $\Omega$ is a holomorphic 2-form which is closed and non-degenerate at every point of $Y_{\mathrm{reg}}$) such that, for any desingularization $f:\ \widetilde{Y} \rightarrow Y$, the 2-form $f^*(\Omega)$ extends to a 2-form on the whole $\widetilde{Y}$, then we say that $Y$ is a \textit{symplectic variety}. Moreover, if $f:\ \widetilde{Y} \rightarrow Y$ is a desingularization such that $f^*(\Omega)$ extends to a symplectic form on $\widetilde{Y}$, then we say that $f$ is a \textit{symplectic desingularization} of $Y$. It must be emphasized that symplectic varieties do not always admit symplectic desingularizations, and when they do, there may be several of them.

As in the introduction, we denote $W=(V \oplus V^*)^{\oplus m}$, we take a reductive algebraic subgroup $G\subset GL(V)$ acting naturally on $W$, and we consider the symplectic reduction $\mu^{-1}(0)/\!/G$. The following conjecture motivates the study (and the name!) of $\mu^{-1}(0)/\!/G$: 
\begin{conjecture}[Kaledin, Lehn, Sorger]  \label{conjKLS}
With the above notation, the irreducible components $Y_1,\ldots,Y_r$ of $\mu^{-1}(0)/\!/G$ are symplectic varieties. Moreover, if every $Y_i$ admits a symplectic desingularization, then the quotient $V^{\oplus m}/\!/ G$ is smooth.  
\end{conjecture}
When $G$ is a finite group, Conjecture \ref{conjKLS} was proved by Kaledin and Verbitsky, but the general case remains open. Let us mention that Becker showed in \cite{Tanja} that the converse of the second part of Conjecture \ref{conjKLS} holds for $G=Sp(V)$ with $\dim(V)=2$. In our setting, that is when $G=GL(V)$, $O(V)$, or $Sp(V)$, one easily checks that Conjecture \ref{conjKLS} holds (see \cite[\S A.2]{Terp} for details).

\subsection{Closures of nilpotent orbits} \label{section22}

We now recall some basic facts concerning the closures of nilpotent orbits (see \cite{CoMc,FuB} for more details). The following well-known result is due to Kostant, Kirillov, Souriau, and Panyushev:

\begin{theoreme} \label{prop1}
The normalization of the closure of an adjoint orbit in a semi-simple Lie algebra is a symplectic variety. 
\end{theoreme}

Let now $\hh$ be a simple Lie algebra of classical type. If $\hh=\sl_{m}$, then every nilpotent element $f \in \sl_{m}$ is conjugate to an element of the form $diag(J_{d_1},\ldots,J_{d_k})$, where each $J_{d_i}$ is a Jordan block of size $d_i$, and $\dd=[d_1,\ldots,d_k]$ is a partition of $m$. Then there exists a one-to-one correspondence between the partitions $\dd=(d_1 \geq \cdots \geq d_k)$ of $m$ and the nilpotent orbits $\OO_\dd \subset \sl_{m}$ (see \cite[\S 3.1]{CoMc}). Now if $\hh=\sp_{2m}$ resp. if $\hh=\so_{2m}$, then a similar description exists (see \cite[\S 5.1]{CoMc}), it is obtained by cutting $\hh$ with a $SL_{2m}$-orbit $\OO_{\dd} \subset \sl_{2m}$, which gives a unique $Sp_{2m}$-orbit resp. a unique $O_{2m}$-orbit, if it is not empty. Let us note that an $O_{2m}$-orbit can be non-connected giving rise to two $SO_{2m}$-orbits that we will denote $\OO_{\dd}^I$ and $\OO_{\dd}^{I\!I}$.

If $\hh=\sl_{m}$, then $\overline{\OO_{\dd}}$ is always normal (\cite{KP2}). For the other classical types, the geometry of $\overline{\OO_{\dd}}$ was studied in \cite{KP3}; in particular, if $\hh=\sp_{2m}$ and $d_1+d_2 \leq 4$ resp. if $\hh=\so_{2m}$ and $d_1 \leq 2$, then $\overline{\OO_{\dd}}$ is normal. In the next sections, we will be interested only by conjugacy classes of elements $f \in \hh$ with $f^2=0$. Hence, from now on, we only consider partitions $\dd$ such that each $d_i \leq 2$. By Theorem \ref{prop1}, the variety $\overline{\OO_{\dd}}$ is symplectic, and we are going to describe its symplectic desingularizations (see \cite{FuB,Fu3,FuNa} for details).

As before, let $\hh$ be a simple Lie algebra of classical type, and let $H$ be the adjoint group of $\hh$. We consider $f:\ Z \rightarrow \overline{\OO_{\dd}}$ a symplectic desingularization. Then, by \cite[Proposition 3.1]{FuB}, the group $H$ acts naturally on $Z$ in such a way that $f$ is $H$-equivariant. One says that $f$ is a \textit{Springer desingularization} if there exists a parabolic subgroup $P \subset H$ and a $H$-equivariant isomorphism between $Z$ and the total space of the cotangent bundle over $H/P$, denoted by $\TT^*(H/P)$. Then, under this isomorphism, the map $f$ becomes  
\begin{equation*}
 \TT^*(H/P) \cong H \times^P \uu \longrightarrow \hh, \hspace{7mm} (h,x) \longmapsto \Ad(h).x,
\end{equation*}
where $\uu$ is the nilradical of the Lie algebra of $P$, and $H \times^P \uu$ denotes the quotient $(H \times \uu)/P$ under the (free) action of $P$ given by $p.(h,u)=(h\circ p^{-1},\Ad(p).u)$.

\begin{theoreme}  \label{ttt} 
\emph{(\cite[Theorem 3.3]{FuB})} With the above notation, if $f:\ Z \rightarrow \overline{\OO_{\dd}}$ is a symplectic desingularization, then $f$ is a Springer desingularization.
\end{theoreme}

Thanks to the work of Fu and Namikawa, the Springer desingularizations of $\overline{\OO_{\dd}}$ are known (up to isomorphism). In particular:

\begin{itemize}
\item Let $\hh=\sl_m$ and $\dd=[2^N,1^{m-2N}]$ for some $0 \leq N \leq \frac{m}{2}$. We denote by $\Gr(p,\CC^{m})$ the Grassmannian of $p$-dimensional subspaces of $\CC^{m}$, and by $\TT_1^*$ resp. by $\TT_2^*$, the cotangent bundle over $\Gr(N,\CC^{m})$ resp. over $\Gr(m-N,\CC^{m})$. By \cite[\S 2]{Fu3}, if $N<\frac{m}{2}$, then $\TT_1^*$ and $\TT_2^*$ are the two Springer desingularizations of $\overline{\OO_{\dd}}$; else, $\TT_1^*=\TT_2^*$ is the unique Springer desingularization of $\overline{\OO_{\dd}}$. 

\item Let $\hh=\sp_{2m}$ and $\dd=[2^N,1^{2(m-N)}]$ for some $0 \leq N \leq m$. Then $\overline{\OO_{\dd}}$ admits a Springer desingularization if and only if $N=m$ (\cite[Proposition 3.19]{FuB}). We denote by $\IG(p,\CC^{2m})$ the Grassmannian of isotropic $p$-dimensional subspaces of $\CC^{2m}$, and by $\TT^*$ the cotangent bundle over $\IG(m,\CC^{2m})$. By \cite[Proposition 3.5]{FuNa}, if $N=m$, then $\TT^*$ is the unique Springer desingularization of $\overline{\OO_{\dd}}$.

\item  Let $\hh=\so_{2m}$ and $\dd=[2^N,1^{2(m-N)}]$  for some $0 \leq N \leq m$ with $N$ even. If $N=m$, then one associates to $\dd$ two distinct nilpotent orbits $\OO_{\dd}^{I}$ and $\OO_{\dd}^{I\!I}$. By \cite[Proposition 3.20]{FuB}, the variety $\overline{\OO_{\dd}}$ admits a Springer desingularization if and only if $N \in \{m-1,m\}$. We denote by $\OG(p,\CC^{2m})$ the Grassmannian of isotropic $p$-dimensional subspaces of $\CC^{2m}$. The Grassmannian $\OG(p,\CC^{2m})$ is irreducible except if $p=m$, in which case $\OG(m,\CC^{2m})=OG^I \cup OG^{I\!I}$ is the union of two irreducible components (exchanged by the natural action of $O_{2m}$). We denote by $\TT_I^*$ resp. by $\TT_{I\!I}^{*}$, the cotangent bundle over $OG^I$ resp. over $OG^{I\!I}$.   
If $N=m-1$, then $\TT_I^*$ and $\TT_{I\!I}^*$ are the two Springer desingularizations of $\overline{\OO_{[2^{m-1},1^2]}}$ by \cite[\S 2]{Fu3}.
If $N=m$, then $\TT_I^*$ resp. $\TT_{I\!I}^*$, is the unique Springer desingularization of $\overline{\OO_{[2^{m}]}^{I}}$ resp. of $\overline{\OO_{[2^{m}]}^{I\!I}}$, by \cite[Proposition 3.5]{FuNa}. 
 
\end{itemize}

\section{Case of \texorpdfstring{$GL_n$}{GLn}} \label{posisimpy}

In this section, we denote $V$ and $V'$ two finite dimensional vector spaces, and we take $G=GL(V)$ and $H=GL(V')$, both acting on 
$$W:=\Hom(V',V) \times \Hom(V,V')$$ 
as follows: 
\begin{equation*}
\forall (g,h) \in G \times H,\ \forall (u_1,u_2) \in W,\ (g,h).(u_1,u_2):=(g \circ u_1 \circ h^{-1}, h \circ u_2 \circ g^{-1}).
\end{equation*}
We denote by $\gg$ resp. by $\hh$, the Lie algebra of $G$ resp. of $H$, and $N:=\min \left( \lfloor \frac{m}{2} \rfloor ,n \right )$, where $n:=\dim(V)$, $m:=\dim(V')$, and $\lfloor . \rfloor$ is the lower integer part.

\subsection{The quotient morphism}  \label{sectionavecX}

The two main results of this section are Proposition \ref{descQuotient}, which describes the symplectic reduction $\mu^{-1}(0)/\!/G$, and Corollary \ref{fctH2}, which gives the Hilbert function $h_0$ of the general fibers of the quotient morphism $\nu:\ \mu^{-1}(0) \rightarrow \mu^{-1}(0)/\!/G$.     

We recall that $W$ is equipped with a $G$-invariant symplectic form $\Omega$ defined by:
\begin{equation} \label{defsymp}
\forall (u_1,u_2), (u'_1,u'_2) \in W,\ \Omega((u_1,u_2),(u'_1,u'_2)):=\tr(u'_1 \circ u_2)-\tr(u_1 \circ u'_2),
\end{equation} 
where $\tr$(.) denotes the trace. The corresponding moment map is given by:
\begin{equation*}  
\begin{array}{lccc}
 \mu:\ &    W  & \rightarrow  & \gg^*  \\
       &  (u_1,u_2)  & \mapsto      &  (f \mapsto \tr(u_2 \circ f \circ u_1))
\end{array}
\end{equation*}
and thus the zero fiber of $\mu$ is the $G \times H$-stable subvariety defined by:  
\begin{equation*}
\mu^{-1}(0)=\left\{(u_1,u_2) \in W \ |\ u_1 \circ u_2=0\right\}.
\end{equation*}
Let us determine the irreducible components of $\mu^{-1}(0)$ as well as their dimensions. Let $p \in \{0,\ldots,m\}$; we define the subvariety
\begin{equation}  \label{defXm}
X_p:=\left\{(u_1,u_2)\in W \ \middle| \ 
    \begin{array}{l}
       \Im(u_2) \subset \Ker(u_1);\\
       \rk(u_2) \leq \min(n,p);\\
       \dim(\Ker(u_1)) \geq \max(m-n,p).   
    \end{array}
\right\} \subset \mu^{-1}(0),
\end{equation} 
and we consider the diagram
\begin{equation*}
\xymatrix{   Z_p:=\{(u_1,u_2,L) \in W \times \Gr(p,V')\ \mid  \ \Im(u_2) \subset L \subset \Ker(u_1)\} \ar@{->>}[d]_{p_1} \ar@{->>}[rd]^{p_2} & \\   X_p & \Gr(p,V') }
\end{equation*} 
where the $p_i$ are the natural projections. We fix $L_0 \in \Gr(p,V')$; the second projection equips $Z_p$ with a structure of homogeneous vector bundle over $\Gr(p,V')$ whose fiber over $L_0$ is isomorphic to $F_p:=\Hom(V'/L_0,V) \times \Hom(V,L_0)$. Hence, $Z_p$ is a smooth variety of dimension $p(m-p)+mn$.

\begin{proposition} \label{compirredfibzero2}
The irreducible components of $\mu^{-1}(0)$ are
$$\left\{
    \begin{array}{ll}
        X_0, \ldots ,X_m &\text{ if } m \leq n;\\
        X_{m-n}, \ldots ,X_n &\text{ if } n < m < 2n;\\
        X_n &\text{ if } m \geq 2n;
    \end{array}
\right.$$
where $X_p$ is defined by (\ref{defXm}).
\end{proposition}

\begin{proof}
We have
$$\mu^{-1}(0)=\{(u_1,u_2)\in W \ \mid  \ \Im(u_2) \subset \Ker(u_1) \}= \bigcup_{i=0}^{m} X_i.$$
Furthermore, for every $p \in \{0,\ldots,m\}$, the morphism $p_1$ is surjective and $Z_p$ is irreducible, hence $X_p$ is irreducible.\\
If  $m \geq 2n$, then  
$$\left\{
    \begin{array}{ll}
        X_0 \subset  \cdots \subset X_n;   \\
        X_n = \cdots=X_{m-n};\\
        X_{m-n} \supset  \cdots \supset X_{m};
    \end{array}
\right.$$  
and thus $\mu^{-1}(0)=X_n$.\\
If $m <2n$, then  
$$\left\{
    \begin{array}{ll}
        X_0 \subset  \cdots \subset X_{\max(0,m-n)};   \\
        X_{\min(m,n)} \supset  \cdots \supset X_{m};
    \end{array}
\right.$$  
and one easily checks that there is no other inclusion relation between the $X_p$. 
\end{proof}

\begin{corollaire} \label{fibzero2}
The dimension of $\mu^{-1}(0)$ is 
$$
\dim(\mu^{-1}(0))= \left\{
    \begin{array}{ll}
            nm+\frac{1}{4}{m}^2   &\text{ if }  m < 2n \text{ and $m$ is even;}\\
            nm+\frac{1}{4}({m}^2-1) &\text{ if }  m < 2n \text{ and $m$ is odd;} \\
            2nm-n^2 &\text{ if } m \geq 2n.
    \end{array}
\right.
$$
\end{corollaire}

\begin{proof}
By Proposition \ref{compirredfibzero2}, it suffices to compute the dimension of $X_p$ for some $p$. If $p \leq n$ or $p \geq m-n$, then one may check that the map $p_1:\ Z_p \rightarrow X_p$ is birational, and thus $Q(p):=\dim(X_p)=\dim(Z_p)=p(m-p)+mn$. It remains simply to study the variations of the polynomial $Q$ to obtain the result.   
\end{proof}

We recall that the quotient morphism $W \rightarrow W/\!/G$ is given by $(u_1,u_2) \mapsto u_2 \circ u_1 \in \End(V')=\hh$, by classical invariant theory (see \cite[§9.1.4]{Pro} for instance). Let us now fix $l \in \{0,\ldots,N\}$. We also fix a basis $\BB$ of $V$ resp. $\BB'$ of $V'$, and we introduce some notation that we will use in the proofs of Proposition \ref{descQuotient} and Lemma \ref{fibreUnsymp1}:
\begin{align}
&\bullet \ (u_1^l,u_2^l):= \left( \begin{bmatrix}
0_{l,m-l}  &I_l \\
0_{n-l,m-l}  & 0_{n-l,l} 
\end{bmatrix},  \begin{bmatrix}
I_l  &0_{l,n-l} \\
0_{m-l,l}  & 0_{m-l,n-l} 
\end{bmatrix}   \right) \in W;  \label{ll1}  \\ 
&\bullet \ f_l:= \begin{bmatrix}
0_{l,m-l}  &I_l \\
0_{m-l,m-l}  & 0_{m-l,l} 
\end{bmatrix} \in \hh.  \label{ll2} 
\end{align}
If $\dd$ is a partition of $m$, then we denote by $\OO_{\dd} \subset \hh\cong \gl_{m}$ the corresponding nilpotent orbit (see Section \ref{section22}). 

\begin{proposition}  \label{descQuotient}
The symplectic reduction of $W$ by $G$ is
$\mu^{-1}(0)/\!/G=\overline{\OO_{[2^N,1^{m-2N}]}}$. 
\end{proposition}

\begin{proof}
If $f \in \mu^{-1}(0)/\!/G$, then there exists $(u_1,u_2) \in \mu^{-1}(0)$ such that $f=u_2 \circ u_1$, and thus $f \circ f=(u_2 \circ u_1) \circ (u_2 \circ u_1)=u_2 \circ (u_1 \circ u_2) \circ u_1=0$, whence the inclusion "$\subset$". 
Now, let $f \in \overline{\OO_{[2^N,1^{m-2N}]}}$. Up to conjugation by an element of $H$, we can suppose that $f=f_l$ for some $l \leq N$, where $f_l$ is defined by (\ref{ll2}). But then $u_2^l \circ u_1^l=f_l$ and $u_1^l \circ u_2^l=0$, where $u_1^l$ and $u_2^l$ are defined by (\ref{ll1}), and thus $f \in \mu^{-1}(0)/\!/G$. 
\end{proof} 

\begin{corollaire}
The symplectic reduction $\mu^{-1}(0)/\!/G \subset \hh$ is irreducible and decomposes into $N+1$ orbits for the adjoint action of $H$: 
$$U_i:=\OO_{[2^i,1^{m-2i}]},\ \text{ for } i=0, \ldots, N.$$
\end{corollaire}

The closures of the nilpotent orbits $U_i$ are nested in the following way:
$$\{0\}=\overline{U_0} \subset \cdots \subset \overline{U_N}=\mu^{-1}(0)/\!/G.$$
Hence, $\mu^{-1}(0)/\!/G$ is a symplectic variety (see Section \ref{gen}), of dimension $2N(m-N)$ (\cite[Corollary 6.1.4]{CoMc}), and whose singular locus is $\overline{U_{N-1}}$ (\cite[\S 3.2]{KP4}). \\
By Corollary \ref{fibzero2}, the dimension of the general fibers of the quotient morphism $\nu$ is 
\begin{equation}  \label{dfg}
\left\{
    \begin{array}{ll}
          nm-\frac{1}{4} {m}^2 &\text{ if } m <2n \text{ and $m$ is even; } \\
          nm-\frac{1}{4} ({m}^2-1) &\text{ if } m<2n \text{ and $m$ is odd;}  \\
          n^2 &\text{ if } m \geq 2n.
    \end{array}
\right.
\end{equation}
If $m<2n$, then $N=\lfloor \frac{m}{2} \rfloor$, and we denote 
\begin{equation}  \label{subH}
G':=\left\{ \begin{bmatrix}
M  &0_{n-N,N} \\
0_{N,n-N}  & I_N 
\end{bmatrix},\ M \in GL_{n-N}\right\} \cong GL_{n-N},
\end{equation}
which is a reductive algebraic subgroup of $G \cong GL_n$.

\begin{proposition} \label{fibreUnsymp1}
The general fibers of the quotient morphism $\nu:\ \mu^{-1}(0) \rightarrow \mu^{-1}(0)/\!/G$ are isomorphic to 
$$\left\{
    \begin{array}{ll}
        G    &\text{ if } m \geq 2n;\\
        G/G'  &\text{ if } m <2n \text{ and $m$ is even}; 
      \end{array}
\right.
$$
where $G' \subset G$ is the subgroup defined by (\ref{subH}).
\end{proposition}

\begin{proof}
We first suppose that $m <2n$ and $m$ is even (that is, $N=\frac{m}{2}$). With the notation (\ref{ll1}), and by a result of Luna (see \cite[\S I.6.2.5, Theorem 10]{SB}), we have the equivalence
$$ G.(u_1^N,u_2^N) \text{ is closed in }\mu^{-1}(0) \Leftrightarrow C_G(G').(u_1^N,u_2^N) \text{ is closed in }\mu^{-1}(0).$$
Now $C_G(G')=\left\{ \begin{bmatrix}
M  &0 \\
0  &\lambda I_{n-N} \end{bmatrix},\ M\in GL_{N},\ \lambda \in \Gm \right\}$, where $\Gm$ denotes the multiplicative group. Hence 
$$C_G(G').(u_1^N,u_2^N)=\left\{ \left(\begin{bmatrix}
0   &M \\
0   &0 \end{bmatrix},\begin{bmatrix}
M^{-1} &0 \\
0     &0 \end{bmatrix} \right)\ ,\ M\in {GL}_{N}\right\} \subset \mu^{-1}(0)$$ 
is a closed subset, and thus $G.(u_1^N,u_2^N)$ is the unique closed orbit contained in the fiber $\nu^{-1}(f_N)$, where $f_N$ is defined by (\ref{ll2}). One may check that $\Stab_G((u_1^N,u_2^N))=G'$. Furthermore, $\dim(G/G')=N(2n-N)$, which is also the dimension of the general fibers of $\nu$ by (\ref{dfg}), and thus $\nu^{-1}(f_N) \cong G/G'$.\\ 
We now suppose that $m \geq 2n$ (that is, $N=n$). One may check that $\Stab_G((u_1^n,u_2^n))=Id$, and thus the fiber $\nu^{-1}(f_n)$ contains a unique closed orbit isomorphic to $G$. But $\dim(G)=n^2$ is the dimension of the general fibers of $\nu$ by (\ref{dfg}), hence $\nu^{-1}(f_n) \cong G$. 
\end{proof}

\begin{corollaire} \label{fctH2}
The Hilbert function $h_0$ of the general fibers of the quotient morphism $\nu:\ \mu^{-1}(0) \rightarrow \mu^{-1}(0)/\!/G$ is given by:
$$\forall M \in \Irr(G),\ h_0(M)=\left\{
    \begin{array}{ll}
     \dim(M)  &\text{ if } m \geq 2n;\\
     \dim(M^{G'})    &\text{ if } m <2n \text{ and $m$ is even}; 
      \end{array}
\right.
$$
where $G' \subset G$ is the subgroup defined by (\ref{subH}).
\end{corollaire}

If $m <2n$ and $m$ is odd, then the situation is more complicated (except the case $m=1$ which is trivial) because the general fibers of the quotient morphism $\nu$ are reducible. From now on, we will only consider the cases where either $m \geq 2n$ or $m <2n$, $m$ is even.

\subsection{The reduction principle for the main component} \label{MropRRED}
In this section we prove our most important theoretical result, which is the \textit{reduction principle} (Proposition \ref{reduction3}). Let us mention that a similar reduction principle (but in a different setting) was already obtained in \cite{Terp1}.

The subvariety $\mu^{-1}(0) \subset W$ being $G \times H$-stable, it follows from \cite[Lemma 3.3]{Br} that the invariant Hilbert scheme
\begin{equation*}
\HH:=\Hilb_{h_0}^{G}(\mu^{-1}(0)) 
\end{equation*}
is a $H$-stable closed subscheme of $\Hilb_{h_0}^G(W)$. We denote by $\HHp$ the main component of $\HH$. 
The scheme $\Hilb_{h_0}^G(W)$ was studied in \cite{Terp1}; let us recall
\begin{proposition} \label{moorppgr}  
\emph{(\cite[\S 4.4]{Terp1})} Let $h_0$ be the Hilbert function given by Corollary \ref{fctH2}, and let $H=GL(V')$ acting naturally on $\Gr(m-h_0(V),V'^*) \times \Gr(m-h_0(V^*),V')$. 
Then there exists a $H$-equivariant morphism
$$
\rho :\  \Hilb_{h_0}^G(W)  \rightarrow  \Gr(m-h_0(V),V'^*) \times \Gr(m-h_0(V^*),V')
$$
given on closed points by $[Z]  \mapsto  (\Ker(f_Z^1),\ \Ker(f_Z^2))$, where $f_Z^1:\ V'^* \cong \Mor^G(W,V) \rightarrow \Mor^G(Z,V)$ and $f_Z^2:\ V' \cong \Mor^G(W,V^*) \rightarrow \Mor^G(Z,V^*)$ are the restriction maps. 
\end{proposition}

By Corollary \ref{fctH2}, we have $h_0(V)=h_0(V^*)=N$. We identify $\Gr(m-N,V'^*)$ with $\Gr(N,V')$, and we denote
$$A_i:=\{(L_1,L_2) \in \Gr(N,V') \times \Gr(m-N,V')\ |\ \dim(L_1 \cap L_2)=N-i\}, \text{ for } i=0, \ldots, N.$$
The $A_i$ are the $N+1$ orbits for the action of $H$ on $\Gr(N,V') \times \Gr(m-N,V')$, and 
$$A_0=\overline{A_0} \subset \overline{A_1} \subset \cdots \subset \overline{A_N}=\Gr(N,V') \times \Gr(m-N,V').$$ 
In particular, $A_N$ is the unique open orbit and 
\begin{equation}  \label{flagvar}
A_0=\FF_{N,m-N}:=\{(L_1,L_2) \in \Gr(N,V') \times \Gr(m-N,V')\ |\ L_1 \subset L_2\},
\end{equation}
which is a partial flag variety, is the unique closed orbit. Let
\begin{itemize} \renewcommand{\labelitemi}{$\bullet$}
\item $a_0:=(L_1,L_2) \in A_0$, and $P$ the parabolic subgroup of $H$ stabilizing $a_0$;
\item $W':=\{ (u_1,u_2) \in W\ |\ L_2 \subset \Ker(u_1) \text{ and } \Im(u_2) \subset L_1 \}$, which is a $G \times P$-module contained in $\mu^{-1}(0)$; and 
\item $\HH':=\Hilb_{h_0}^{G}(W')$, and $\HH'^{\mathrm{main}}$ its main component.
\end{itemize} 
If either $m \geq 2n$ or $m<2n$, $m$ even, then $h_0$ coincides with the Hilbert function of the general fibers of the quotient morphism $W' \rightarrow W'/\!/G$ by \cite[Proposition 4.13]{Terp1}; in particular, $\HH'^{\mathrm{main}}$ is well-defined. We are going to prove 

\begin{proposition} \label{reduction3}
If either $m \geq 2n$ or $m<2n$, $m$ even, and with the above notation, there is a $H$-equivariant isomorphism
$$ \HHp \cong H {\times}^{P} \HH'^{\mathrm{main}}.$$ 
\end{proposition} 

First of all, we need

\begin{lemme}  \label{versX0}
If either $m \geq 2n$ or $m<2n$ with $m$ even, then the morphism $\rho$ of Proposition \ref{moorppgr} sends $\HHp$ onto $A_0$, the $H$-variety defined by (\ref{flagvar}).  
\end{lemme}

\begin{proof}
As the quotient morphism $\nu:\ \mu^{-1}(0) \rightarrow \mu^{-1}(0)/\!/G$ is flat over the open orbit $U_N$, the restriction of the Hilbert-Chow morphism $\gamma$ to $\gamma^{-1}(U_N)$ is an isomorphism. We fix $f_N \in U_N$, and we denote $Q:=\Stab_{H}(f_N)$, and $[Z_N]$ the unique point of $\HH$ such that $\gamma([Z_N])=f_N$. As $\gamma$ is $H$-equivariant, $[Z_N]$ is $Q$-stable. In addition, $\rho$ is also $H$-equivariant, hence $\rho([Z_N])$ is a fixed point for the action of $Q$. But one may check that $\Gr(N,V') \times \Gr(m-N,V')$ has a unique fixed point for $Q$, which is contained in $A_0$. Then, as $A_0$ is $H$-stable, we have $\rho([Z]) \in A_0$, for every $[Z] \in \gamma^{-1}(U_N)$. Hence, $\rho^{-1}(A_0)$ is a closed subscheme of $\Hilb_{h_0}^{G}(W)$ containing $\gamma^{-1}(U_N)$, and the result follows.
\end{proof}

The restriction $\rho_{|\HHp}:\ \HHp \rightarrow A_0$ is $H$-equivariant, hence $\HHp$ is the total space of a $H$-homogeneous fiber bundle over $A_0$. Let $F$ be the scheme-theoretic fiber of $\rho_{|\HHp}$ over $a_0$. The action of $P$ on $\HHp$, induced by the action of $H$, stabilizes $F$, and there is a $H$-equivariant isomorphism   
\begin{equation} \label{ut1}
\HHp \cong H {\times}^{P} F.
\end{equation}
Hence, to prove Proposition \ref{reduction3}, we have to determine $F$ as a $P$-scheme. We start by considering $F'$, the scheme-theoretic fiber of the restriction $\rho_{|\HH}:\ \HH \rightarrow  \Gr(N,V') \times \Gr(m-N,V')$ over $a_0$, as a $P$-scheme. The proof of the next lemma is analogous to the proof of \cite[Lemma 3.7]{Terp1}.

\begin{lemme} \label{fibrehil2s}
With the above notation, there is a $P$-equivariant isomorphism 
$$F' \cong \HH',$$
where $P$ acts on $\HH'$ via its action on $W'$. 
\end{lemme}

As $\HHp$ is an irreducible variety of dimension $2N(m-N)$, we deduce from (\ref{ut1}) that $F$ is an irreducible variety of dimension $N^2$. By Lemma \ref{fibrehil2s}, the fiber $F$ is isomorphic to a subvariety of $\HH'^{\mathrm{main}}$, but $\dim(\HH'^{\mathrm{main}})=N^2$, and thus there is a $P$-equivariant isomorphism 
\begin{equation}  \label{moniox2}
F \cong \HH'^{\mathrm{main}},
\end{equation}
and Proposition \ref{reduction3} follows.

\begin{remarque}
The scheme $\HH'$ is $P$-stable and identifies with a closed subscheme of $\HH$, hence there is an inclusion of $H$-schemes $H \times^P \HH' \subset \HH$. 
\end{remarque}

\subsection{Proofs of Theorems A and B for \texorpdfstring{$GL(V)$}{GL(V)}}  \label{proofGLn}
Our strategy to prove Theorems A and B is the following: first we perform a reduction step (Proposition \ref{reduction3}), then we use \cite[\S 1, Theorem]{Terp1} to identify $\Hilb_{h_{W'}}^{G}(W')^{\mathrm{main}}$, and finally we compare the Hilbert-Chow morphism $\gamma:\ \Hilb_{h_0}^{G}(\mu^{-1}(0))^{\mathrm{main}} \rightarrow \mu^{-1}(0)/\!/G$ with the Springer desingularizations of $\mu^{-1}(0)/\!/G$. Let us start by recalling

\begin{theoreme} \emph{(\cite[\S 1, Theorem]{Terp1})} \label{Terpy}
Let $G=GL(V)$, let $W=\Hom(V',V) \times \Hom(V,V')$, and let $h_W$ be the Hilbert function of the general fibers of the quotient morphism $W \rightarrow W/\!/G$. We denote $n:=\dim(V)$, $m:=\dim(V')$, and by $Y_0$ the blow-up of $W/\!/G=\End(V')^{\leq n}:=\{ f \in \End(V')\ |\ \rk(f) \leq n\}$ at $0$. In the following cases, the invariant Hilbert scheme $\HH':=\Hilb_{h_W}^{G}(W)$ is a smooth variety and the Hilbert-Chow morphism is the succession of blow-up described as follows:
\begin{itemize}
\item if $n \geq 2m-1$, then $\HH'\cong W/\!/G=\End(V')$;
\item if $m>n=1$ or $m=n=2$, then $\HH' \cong Y_0$;
\item if $m>n=2$, then $\HH'$ is isomorphic to the blow-up of $Y_0$ along the strict transform of $\End(V')^{\leq 1}$.
\end{itemize}
\end{theoreme}

Let us now consider the following diagram
\begin{equation*}
\xymatrix{ & \FF_{N,m-N} \ar@{->>}[ld]_{p_1} \ar@{->>}[rd]^{p_2} \\   \Gr(N,V') && \Gr(m-N,V') }
\end{equation*}
where $\FF_{N,m-N}$ is defined by (\ref{flagvar}), $p_1$ and $p_2$ being the natural projections. We denote by $\underline{V'}$ the constant vector bundle over $\FF_{N,m-N}$ with fiber $V'$, and by $T_1$ resp. by $T_2$, the pull-back of the tautological bundle over $\Gr(N,V')$ by $p_1$, resp. over $\Gr(m-N,V')$ by $p_2$. In particular, if $N=\frac{m}{2}$, then $\FF_{N,m-N}=\Gr(N,V')$ and $T:=T_1=T_2$ is the tautological bundle over $\Gr(N,V')$.

We deduce from Proposition \ref{reduction3} and Theorem \ref{Terpy} the following $H$-equivariant isomorphisms 
\begin{equation} \label{descHpGl}
\HHp \cong \left\{
    \begin{array}{ll}
       \Hom(\underline{V'}/T,T)             &\text{ if } n \geq m-1 \text{ and $m$ is even}; \\
       \Hom(\underline{V'}/T_2,T_1)         &\text{ if } n=1 \text{ and } m \geq 3;\\
        Bl_0(\Hom(\underline{V'}/T_2,T_1))  &\text{ if } n=2 \text{ and } m \geq 4;
    \end{array}
\right.
\end{equation}
where $Bl_0(.)$ denotes the blow-up along the zero section. In all these cases, $\HHp$ is smooth, and thus the Hilbert-Chow morphism $\gamma:\ \HHp \rightarrow \mu^{-1}(0)/\!/G$ is a desingularization.

On the other hand, we saw in Section \ref{gen} that the Springer desingularizations of $\mu^{-1}(0)/\!/G$ are the cotangent bundles $\TT_1^*:=\TT^* \Gr(N,V')$ and $\TT_2^*:=\TT^* \Gr(m-N,V') \cong \TT^* \Gr(N,{V'}^*)$. We then distinguish between two cases:
\begin{enumerate}

\item If $N<\frac{m}{2}$, then let us prove by contradiction that $\gamma:\ \HHp \rightarrow \mu^{-1}(0)/\!/G$ cannot be a Springer desingularization. First, we consider the isomorphism of $G \times H$-modules $W \cong W^*$.
Denoting $\HH^*:=\Hilb_{h_0}^{G}(\mu^{*-1}(0))$, where $\mu^*$ is the moment map for the natural action of $G$ on $W^*$, there is an isomorphism of $H$-varieties $\HHp \cong \HH^{* \mathrm{main}}$. Now if we suppose that (say) $\HHp \cong \TT_1^*$, then we get that $\HH^{* \mathrm{main}} \cong \TT_2^*$, and thus $\TT_1^* \cong \TT_2^*$ as a $H$-variety, which is absurd.\\
However, one easily checks that if $n \in \{1,2\}$ and $m \geq 2n+1$, then $\gamma:\ \HHp \rightarrow \mu^{-1}(0)/\!/G$ dominates the two Springer desingularizations $\TT_1^*$ and $\TT_2^*$ (see \cite[\S A.2.2]{Terp} for details).

\item If $N=\frac{m}{2}$, then $\TT^*:=\TT_1^*=\TT_2^*$ is the unique Springer desingularization of $\mu^{-1}(0)/\!/G$. Let us show that $\gamma:\ \HHp \rightarrow \mu^{-1}(0)/\!/G$ is the Springer desingularization if and only if $n \geq m-1$. The implication "$\Leftarrow$" is given by (\ref{descHpGl}) since $\TT^* \cong \Hom(\underline{V'}/T,T)$. The other implication is given by: 

\begin{lemme}
If $N=\frac{m}{2}$ and the Hilbert-Chow morphism $\gamma:\ \HHp \rightarrow \mu^{-1}(0)/\!/G$ is the Springer desingularization, then $n \geq m-1$. 
\end{lemme}
   
\begin{proof}
We suppose that $\gamma:\ \HHp \rightarrow \mu^{-1}(0)/\!/G$  is the Springer desingularization, that is, $\HHp \cong \TT^*$  as a $H$-variety. We fix $L \in \Gr(N,V')$, and we define $P \subset H$, $W'$, and $\HH'^{\mathrm{main}}$ as in Section \ref{MropRRED}. We have $\TT^* \cong H \times^P \Hom(V'/L,L)$, and it follows from (\ref{ut1}) and (\ref{moniox2}) that $\HHp \cong H \times^P \HH'^{\mathrm{main}}$. Hence, $\HH'^{\mathrm{main}} \cong \Hom(V'/L,L)$ as a $P$-variety. We denote by $\gamma':\ \HH'^{\mathrm{main}} \rightarrow W'/\!/G$ the restriction of the Hilbert-Chow morphism. As $\gamma'$ is projective and birational, and $W'/\!/G=\Hom(V'/L,L)$ is smooth, Zariski's Main Theorem implies that $\gamma'$ is an isomorphism. It follows that the quotient morphism $\nu':\ W' \rightarrow W'/\!/G$ is flat, and thus $n \geq 2N-1$ by \cite[Corollary 4.12]{Terp1}.
\end{proof}

In addition, if $m=4$ and $n=2$, then by (\ref{descHpGl}) we have $\HHp \cong Bl_0(\TT^*)$, and thus $\gamma$ dominates the unique Springer desingularization of $\mu^{-1}(0)/\!/G$. 
\end{enumerate}

\subsection{Reducibility of the invariant Hilbert scheme}  \label{reductibilité_cas_symp}
The aim of this section is to prove Proposition C from the introduction, for $G=GL(V)$. We suppose that $m \geq 2n$, then $N=n$. We fix
\begin{equation} \label{defxn}
a_n=(L'_1,L'_2) \in A_n
\end{equation} 
a point of the open $H$-orbit of $\Gr(n,V') \times \Gr(m-n,V')$, and we consider 
\begin{align*}
W^{\prime \prime}&:= \{ (u_1,u_2) \in W \ |\ L'_2 \subset \Ker(u_1) \text{ and } \Im(u_2) \subset L'_1 \}\\
   &\cong \Hom(V'/L'_2,V) \times \Hom(V,L'_1),
\end{align*}   
which is a $G$-submodule of $W$. As $V'=L'_1 \oplus L'_2$, there is a natural identification $W^{\prime \prime}\cong \Hom(L'_1,V) \times \Hom(V,L'_1)$ as a $G$-module. Hence, the  $G$-module $W^{\prime \prime}$ is symplectic and we denote by $\mu^{\prime \prime}:\ W^{\prime \prime} \rightarrow \gg^*$ the corresponding $G$-equivariant moment map (see the beginning of Section \ref{sectionavecX} for details). The proof of the next lemma is analogous to the proof of \cite[Lemma 3.7]{Terp1}. 

\begin{lemme} \label{fibrehil3}
We suppose that $m \geq 2n$, and let $\rho:\ \HH \rightarrow \Gr(n,V') \times \Gr(m-n,V')$ be the morphism of Proposition \ref{moorppgr}. The scheme-theoretic fiber $F^{\prime \prime}$ of $\rho$ over the point $a_n$, defined by (\ref{defxn}), is isomorphic to the invariant Hilbert scheme ${\Hilb}_{h_0}^{G} (\mu^{\prime \prime -1}(0))$, where $h_0$ is the Hilbert function defined by $h_0(M)=\dim(M)$, for every $M \in \Irr(G)$, and $\mu'':\ W'' \rightarrow \gg^*$ is the moment map defined above.
\end{lemme}

\begin{remarque}
The Hilbert function $h_0$ of Lemma \ref{fibrehil3} does not generally coincide with the Hilbert function of the general fibers of the quotient morphism $\mu^{\prime \prime -1}(0) \rightarrow \mu^{\prime \prime -1}(0)/\!/G$. 
\end{remarque}

By Lemma \ref{versX0}, the morphism $\rho:\ \Hilb_{h_0}^G(W) \rightarrow \Gr(n,V') \times \Gr(m-n,V')$ of Proposition \ref{moorppgr} sends $\HHp$ onto $A_0$. Hence, to prove Proposition C for $G=GL(V)$, it is enough, by Lemma \ref{fibrehil3}, to prove that ${\Hilb}_{h_0}^{G} (\mu^{\prime \prime -1}(0))$ is non-empty. \\
We denote $V'':=L'_1$, and we equip $W'' \cong \Hom(V'',V) \times \Hom(V,V'')$ with the natural action of $H':=GL(V'')$. Then 
\begin{align*}  
{\CC[W'']}_2 & \cong (S^2(V'') \otimes S^2(V^{*})) \oplus (S^2(V''^*) \otimes S^2(V))  \\
       &\ \oplus  ({\Lambda}^2 (V'') \otimes {\Lambda}^2 (V^{*})) \oplus  ({\Lambda}^2 (V''^*) \otimes {\Lambda}^2 (V)) \\
       &\ \oplus ((sl(V'') \oplus M_0) \otimes (sl(V) \oplus V_0)) \text{ as a $G \times H'$-module,} 
\end{align*} 
where $V_0$ is the trivial $G$-module resp. $M_0$, is the trivial $H'$-module, and $sl(V''):=\{ f \in \End(V'')\ \mid  \ \tr(f)=0 \}$.\\
We denote by $I_0$ the ideal of $\CC[W'']$ generated by $(sl(V'') \otimes V_0) \oplus (M_0 \otimes V_0) \oplus (M_0 \otimes sl(V)) \subset {\CC[W'']}_2$. The ideal $I_0$ is homogeneous, $G \times H'$-stable, and contains the ideal generated by the homogeneous $H'$-invariants of positive degree of $\CC[W'']$. In particular, $I_0$ identifies with an ideal of $\CC[\mu^{\prime \prime -1}(0)]$.  

\begin{proposition}   \label{pppfixe}
Let $I_0 \subset \CC[W'']$ be the ideal defined above, then $I_0$ is a point of the invariant Hilbert scheme ${\Hilb}_{h_0}^{G} (\mu^{\prime \prime -1}(0))$ defined in Lemma \ref{fibrehil3}.
\end{proposition}

\begin{proof}
We have to check that the ideal $I_0$ has the Hilbert function $h_0$, that is, 
$$\CC[W'']/I_0 \cong \bigoplus_{M \in \Irr(G)}  M^{\oplus \dim(M)}$$
as a $G$-module. To do that that, we are going to adapt the method used by Kraft and Schwarz to prove \cite[Theorem 9.1]{KS}. The result [loc. cit.] was used in \cite[\S 2.1.3 and \S 3.3.2]{Terp}.\\
 We denote $R:= V^{\prime \prime *} \otimes V$, which is an irreducible $G \times H'$-submodule of $W^{\prime \prime *} \cong R \oplus R^*$. Then $R$ and $R^*$ are orthogonal modulo $I_0$, which means that the image of the $G \times H'$-submodule $R \otimes R^* \subset \CC[W'']_2$ in $\CC[W'']/I_0$ is isomorphic to the highest weight component of $R \otimes R^*$ (that is, $sl(V'') \otimes sl(V)$). Then, by \cite[Lemme 4.1]{Br5}, any irreducible $G \times H'$-submodule of $\CC[R]$ is orthogonal to any irreducible $G \times H'$-submodule of $\CC[R^*]$, and thus the natural morphism
$$ \phi:\ \CC[R]^{U \times U'} \otimes \CC[R^*]^{U \times U'} \rightarrow (\CC[W'']/I_0)^{U \times U'} $$
is surjective, where $U$ resp. $U'$, denotes the unipotent radical of a Borel subgroup $B \subset G$ resp. $B' \subset H'$. Furthermore, if $T \subset B$ resp. if $T' \subset H'$, is a maximal torus, then $\phi$ is $T \times T'$-equivariant.\\
Now by \cite[\S 13.5.1]{Pro} we have the following isomorphisms of $T \times T'$-algebras $\CC[R]^{U \times U'} \cong \CC[x_1,\ldots,x_n]$, where $x_i \in \Lambda^i V'' \otimes \Lambda^i V^*$ is a highest weight vector, and $\CC[R^*]^{U \times U'} \cong \CC[y_1,\ldots,y_n]$, where $y_j \in \Lambda^j {V''}^* \otimes \Lambda^j V$ is a highest weight vector. Hence, there is an exact sequence
$$ 0 \rightarrow K_0 \rightarrow \CC[x_1,\ldots,x_n,y_1,\ldots,y_n] \rightarrow (\CC[W'']/I_0)^{U \times U'} \rightarrow 0,$$
where $K_0$ is the kernel of $\phi$. One may check that the ideal $K_0$ is generated by the products $x_r y_s$ with $r+s>n$ (see \cite[\S 9, Proof of Theorem 9.1(1)]{KS}).\\
We denote $\Lambda=\left \langle \epsilon_1,\ldots , \epsilon_n \right \rangle$ the weight lattice of the linear group $GL_n$ with its natural basis, and $\Lambda_+ \subset \Lambda$ the subset of dominant weights, that is, weights of the form $r_1 \epsilon_1+\ldots+r_n \epsilon_n$, with $r_1 \geq \ldots \geq r_n$. If $\lambda \in \Lambda_+$, then we denote by $S^{\lambda}(\CC^n)$ the irreducible $GL_n$-module of highest weight $\lambda$. We fix $\lambda=k_1 \epsilon_1+\ldots+k_t \epsilon_t-k_{t+1} \epsilon_{t+1}-\ldots-k_n \epsilon_n \in \Lambda_+$, where each $k_i$ is a nonnegative integer. One easily checks that the weight of the monomial 
$$  x_{n-t}^{k_{t+1}} x_{n-t-1}^{k_{t+2}-k_{t+1}} x_{n-t-2}^{k_{t+3}-k_{t+2}}  \cdots x_{1}^{k_{n}-k_{n-1}} y_{t}^{k_{t}} y_{t-1}^{k_{t-1}-k_{t}} y_{t-2}^{k_{t-2}-k_{t-1}}  \cdots y_{1}^{k_{1}-k_{2}}$$ 
for the action of $T \times T'$ is $(\lambda,\lambda^*)$, where $\lambda^*$ denotes the highest weight of the $GL_n$-module $S^{\lambda}({\CC^n}^*)$, and that $\lambda$ uniquely determines this monomial. We get that the isotypic component of the $G$-module $S^{\lambda}(V)$ in $\CC[W'']/I_0$ is the $G \times H'$-module $S^{\lambda}(V^{\prime \prime *}) \otimes S^{\lambda}(V)$. As $\dim(V)=\dim(V'')=n$, we have $\dim(S^{\lambda}(V))=\dim(S^{\lambda}(V^{\prime \prime *}))$, for every $\lambda \in \Lambda_+$. In other words, each irreducible $G$-module $M$ occurs in $\CC[W'']/I_0$ with multiplicity $\dim(M)$. 
\end{proof}

By Proposition \ref{pppfixe}, the scheme ${\Hilb}_{h_0}^{G} (\mu'^{-1}(0))$ is non-empty, and thus $\HH$ has an irreducible component, different from $\HHp$, of dimension greater or equal to $\dim(A_n)=2n(m-n)$, which implies Proposition C for $G=GL(V)$.

\subsection{Study of the case \texorpdfstring{$n=1$}{n=1}} \label{GLsympN1}

We saw in Section \ref{proofGLn} that $\HHp$ is a smooth variety, and in Section \ref{reductibilité_cas_symp} that $\HH$ is always reducible. In this section, we determine the irreducible components of $\HH$ when $n=1$.

We suppose that $m \geq 2$ (the case $m=1$ being trivial). Then $G=\Gm$ is the multiplicative group, $\mu^{-1}(0)/\!/G=\overline{\OO_{[2,1^{m-2}]}} \subset \hh$ by Proposition \ref{descQuotient}, and the morphism of Proposition \ref{moorppgr} is $\rho:\ \Hilb_{h_0}^G(W) \rightarrow \PP(V') \times \PP(V'^*)$. The Segre embedding gives a $H$-equivariant isomorphism $\PP(V') \times \PP(V'^*) \cong \PP(\hh^{\leq 1})$, where $\hh^{\leq 1}:=\{ f \in \hh \ |\ \rk(f) \leq 1\}$, and thus we can consider $\rho':\ \Hilb_{h_0}^G(W) \rightarrow \PP(\hh^{\leq 1})$, the morphism induced by $\rho$.

\begin{proposition}  \label{casSympn1}
We equip the invariant Hilbert scheme $\HH$ with its reduced structure. If $m>n=1$, then there is a $H$-equivariant isomorphism
$$\HH \cong \left \{(f,L) \in \overline{\OO_{[2,1^{m-2}]}} \times \PP (\hh^{\leq 1}) \ \mid  \ f \in L  \right \}.$$
In particular, $\HH$ is the union of two smooth irreducible components of dimension $2m-2$ defined by:
\begin{itemize} \renewcommand{\labelitemi}{$\bullet$}
\item $C_1:=\left \{(f,L) \in \overline{\OO_{[2,1^{m-2}]}} \times \PP (\overline{\OO_{[2,1^{m-2}]}}) \ \mid  \ f \in L \right \}=\HHp$, and the Hilbert-Chow morphism $\gamma:\ \HHp \rightarrow \overline{\OO_{[2,1^{m-2}]}}$ is the blow-up of $\overline{\OO_{[2,1^{m-2}]}}$ at $0$; 
\item $C_2:=\left \{(0,L) \in \overline{\OO_{[2,1^{m-2}]}} \times \PP ({\hh}^{\leq 1}) \right \} \cong  \PP ({\hh}^{\leq 1})$, and the Hilbert-Chow morphism is the zero map.
\end{itemize}
\end{proposition}

\begin{proof}
By \cite[\S 1,Theorem]{Terp1}, there is a $H$-equivariant isomorphism 
$$\gamma \times \rho':\ \Hilb_{h_0}^{G}(W) \rightarrow  \left \{(f,L) \in {\hh}^{\leq 1} \times \PP ({\hh}^{\leq 1}) \ \mid  \ f \in L  \right \}.$$
Since $\HH \hookrightarrow \Hilb_{h_0}^{G}(W)$, there is a $H$-equivariant closed embedding 
$$\gamma \times \rho':\ \HH \hookrightarrow \YY:=\left \{(f,L) \in \overline{\OO_{[2,1^{m-2}]}} \times \PP ({\hh}^{\leq 1}) \ \mid  \ f \in L  \right \}.$$ 
One may check that $\YY$ is the union of the two irreducible components $C_1$ and $C_2$, both of dimension $2m-2$. The morphism $\gamma \times \rho'$ sends $\HHp$ into $C_1$; the varieties $\HHp$ and $C_1$ have the same dimension, hence $\gamma \times \rho':\ \HHp \rightarrow C_1$ is an isomorphism. On the other hand, we saw in Section \ref{reductibilité_cas_symp} that $\HH$ admits another irreducible component, denoted by $\HH_2$, of dimension at least $2m-2$, which is the dimension of $C_2$, and thus $\gamma \times \rho'$ is an isomorphism between $\HH_2$ and $C_2$. 
\end{proof}

\begin{remarque}
One may check that the component $C_2$ of Proposition \ref{casSympn1} consists of the homogeneous ideals of $\CC[\mu^{-1}(0)]$.
\end{remarque}

When $m \geq 2n \geq 4$, irreducible components of dimension greater than $\dim(\HHp)$ may appear. For instance, if $n=2$ and $m \geq 4$, then one may check that the irreducible component consisting of the homogeneous ideals of $\CC[\mu^{-1}(0)]$ is of dimension $4m-5$, whereas the main component $\HHp$ is of dimension $4m-8$. In addition, we showed in Section \ref{reductibilité_cas_symp} that $\HH$ has at least two components, but $\HH$ may have more components.

\section{Case of \texorpdfstring{$Sp_n$}{Spn}} \label{posisimpy2}

Let $V$ and $V'$ be two vector spaces of dimension $n$ (which is even) and $m$ respectively, and let $W:=\Hom(V',V) \times \Hom(V,V')$. We denote $E:=V' \oplus V'^*$ on which we fix a non-degenerate quadratic form $q$, and we take $G=Sp(V)$ and $H=SO(E)$. As $G$ resp. $H$, preserves a non-degenerate bilinear form on $V$ resp. on $E$, we can identify $V \cong V^*$ as a $G$-module resp. $E \cong E^*$ as a $H$-module. It follows that
\begin{align*}
W &\cong \Hom(V',V) \times \Hom(V'^*,V^*) \\
  &\cong \Hom(V',V) \times \Hom(V'^*,V) \\
  &\cong \Hom(E,V)
\end{align*}
as a $G$-module, and thus $H$ acts naturally on $W$. We denote by $\gg$ resp. by $\hh$, the Lie algebra of $G$ resp. of $H$.

\subsection{The quotient morphism}  \label{morppquotientSpn}

The main results of this section are Proposition \ref{descqu}, which describes the irreducible components of the symplectic reduction $\mu^{-1}(0)/\!/G$, and Corollary \ref{fctHSp2}, which gives the Hilbert function of the general fibers of the quotient morphism $\nu:\ \mu^{-1}(0) \rightarrow \mu^{-1}(0)/\!/G$ for each irreducible component of $\mu^{-1}(0)/\!/G$. Contrary to the case of $GL(V)$ studied in Section \ref{posisimpy}, we will see that $\mu^{-1}(0)/\!/G$ is reducible when $m \leq n$ and $m$ is even.

We have seen that $W$ is equipped with a $G$-invariant symplectic form (see the beginning of Section \ref{sectionavecX} for details). If $w \in \Hom(E,V)$, we denote the transpose of $w$ by $\leftexp{t}{w} \in \Hom(V^*,E^*)\cong \Hom(V,E)$. Then, by \cite[Proposition 3.1]{Tanja}, the zero fiber of the moment map $\mu:\ W \rightarrow \gg^*$ is the $G \times H$-stable subvariety defined by: 
$$\mu^{-1}(0)= \{ w \in W \ |\ w \circ \leftexp{t}{w}=0\}.$$

\begin{remarque}  \label{veryuseful}
One may check that the biggest subgroup of $GL(E)$ that stabilizes $\mu^{-1}(0)$ in $W$ is the orthogonal group $O(E)$. However, we prefer to consider the action of $H=SO(E)$ for practical reasons. 
\end{remarque}

The proof of the next proposition is analogous to those of Proposition \ref{compirredfibzero2} and Corollary \ref{fibzero2}. 

\begin{proposition} \label{dimmudezero33}
The zero fiber of the moment map $\mu:\ W \rightarrow \gg^*$ is 
\begin{itemize} 
\item an irreducible subvariety of dimension $2mn-\frac{1}{2}n(n+1)$ if $m>n$;
\item the union of two irreducible components of dimension $mn+\frac{1}{2}m(m-1)$ if $m \leq n$.
\end{itemize}
\end{proposition}

If $\dd$ is a partition of $2m$, then we denote by $\OO_{\dd}$ resp. by $\OO_{\dd}^{I}$ and $\OO_{\dd}^{I\!I}$, the corresponding nilpotent orbit(s) of $\hh \cong \so_{2m}$ associated to $\dd$ (see Section \ref{section22}). The following result was proved by Becker: 

\begin{proposition} \emph{(\cite[Proposition 3.6]{Tanja})} \label{descqu}
The symplectic reduction of $W$ by $G$ is 
$$
\mu^{-1}(0)/\!/G = \left\{
    \begin{array}{ll}
        \overline{\OO_{[2^n,1^{2(m-n)}]}} & \text{ if } m>n;\\
        \overline{\OO_{[2^{m-1},1^2]}} & \text{ if } m < n \text{ and } m \text{ is odd};\\
        \overline{\OO_{[2^{m}]}^{I}} \cup \overline{\OO_{[2^{m}]}^{I\!I}} & \text{ if } m \leq n \text{ and } m \text{ is even}.
    \end{array}
\right.
$$
\end{proposition}

\begin{corollaire}
The orbits for the adjoint action of $H$ on $\mu^{-1}(0)/\!/G$ are
\begin{itemize} \renewcommand{\labelitemi}{$\bullet$}
\item $U_i:=\OO_{[2^i,1^{2(m-i)}]}$, for $i=0,2, \ldots,n$,  if $m>n$; 
\item $U_i:=\OO_{[2^i,1^{2(m-i)}]}$, for $i=0,2, \ldots, m-1$, if $m<n$ and $m$ is odd;
\item $U_i:=\OO_{[2^i,1^{2(m-i)}]}$, for $i=0,2, \ldots, m-2$, and $U_{m}^{I}:=\OO_{[2^{m}]}^{I}$, $U_{m}^{I\!I}:=\OO_{[2^{m}]}^{I\!I}$, if $m \leq n$ and $m$ is even.
\end{itemize}
\end{corollaire}

The closures of the nilpotent orbits $U_i$ are nested in the following way:
\begin{equation*}
\left\{
    \begin{array}{ll}
        \{0\}=\overline{U_0} \subset \overline{U_2} \subset \cdots  \subset \overline{U_n} &\text{ if $m>n$;} \\
         \{0\}=\overline{U_0} \subset \overline{U_2} \subset \cdots \subset \overline{U_{m-1}} &\text{ if $m<n$ and $m$ is odd;} \\
         \{0\}=\overline{U_0} \subset \overline{U_2} \subset \cdots \subset \overline{U_{m-2}}=\overline{U_{m}^{I}} \cap \overline{U_{m}^{I\!I}}&\text{ if $m \leq n$ and $m$ is even.}    
    \end{array}
\right.
\end{equation*}

If $m>n$ or $m$ is odd resp. if $m \leq n$ and $m$ is even, then the symplectic reduction $\mu^{-1}(0)/\!/G$  is the closure of a nilpotent orbit resp. the union of two closures of nilpotent orbits, and thus the irreducible components of $\mu^{-1}(0)/\!/G$ are symplectic varieties (see Section \ref{gen}). If $m>n$, then $\mu^{-1}(0)/\!/G$ is of dimension $2mn-n(n+1)$, and its singular locus is $\overline{U_{n-2}}$. On the other hand, if $m \leq n$, then each irreducible component of $\mu^{-1}(0)/\!/G$ is of dimension $m(m-1)$, and the singular locus of $\mu^{-1}(0)/\!/G$ is $\overline{U_{m-2}}$ resp. $\overline{U_{m-3}}$, when $m$ is even resp. when $m$ is odd. The dimension of the irreducible components of $\mu^{-1}(0)/\!/G$ is given by \cite[Corollary 6.1.4]{CoMc}, and the singular locus of $\mu^{-1}(0)/\!/G$ is given by \cite[Theorem 2]{KP3}.

We are now interested in the Hilbert function of the general fibers of the quotient morphism for each irreducible component of $\mu^{-1}(0)/\!/G$. We will distinguish between the following cases: 
\begin{itemize}
\item If $m>n$, then $\mu^{-1}(0)/\!/G$ is irreducible, and we denote by $h_0$ the Hilbert function of the general fibers of the quotient morphism $\nu:\ \mu^{-1}(0) \rightarrow \mu^{-1}(0)/\!/G$. By Proposition \ref{dimmudezero33}, the dimension of these fibers is $\frac{1}{2}n(n+1)$. 
\item If $m \leq n$ and $m$ is even, then by Proposition \ref{dimmudezero33}, the zero fiber $\mu^{-1}(0)$ is the union of two irreducible components that we denote by $X_I$ and by $X_{I\!I}$. Let $\nu_I:\ X_I \rightarrow Y_I$ and $\nu_{I\!I}:\ X_{I\!I} \rightarrow Y_{I\!I}$ be the quotient morphisms. Up to the exchange of $X_I$ and $X_{I\!I}$, we can suppose that $Y_I=\overline{U_{m}^{I}}$ and $Y_{I\!I}=\overline{U_{m}^{I\!I}}$. The orthogonal group $O(E)$ acts transitively on $U_{m}^{I} \cup U_{m}^{I\!I}$, hence the general fibers of $\nu_{I}$ and $\nu_{I\!I}$ are isomorphic. In particular, these fibers have the same Hilbert function, denoted by $h_0$, and the same dimension, which is $mn-\frac{1}{2}m(m-1)$.  
\item If $m <n$ and $m$ is odd, then $\mu^{-1}(0)/\!/G$ is irreducible, and we denote by $h_0$ the Hilbert function of the general fibers of the quotient morphism $\nu:\ \mu^{-1}(0) \rightarrow \mu^{-1}(0)/\!/G$. These fibers being reducible, determining $h_0$ is is more complicated than in the previous cases (except the case $m=1$ which is trivial). From now on, we will always exclude the case where $m <n$ and $m$ is odd.
\end{itemize}

If $m < n$ and $m$ is even, then we denote 
\begin{equation}  \label{defHsymp}
G':=\left\{ \begin{bmatrix}
M  &0_{n-m,m} \\
0_{m,n-m}  & I_{m} 
\end{bmatrix},\ M \in Sp_{n-m} \right\} \cong Sp_{n-m},
\end{equation}
which is a reductive algebraic subgroup of $G \cong Sp_n$. The proof of the next proposition is analogous to the proof of Proposition \ref{fibreUnsymp1}:

\begin{proposition} \label{fibreUnsymp}
If $m>n$, then the general fibers of the quotient morphism $\nu:\ \mu^{-1}(0) \rightarrow \mu^{-1}(0)/\!/G$ are isomorphic to $G$.\\
If $m=n$, then the general fibers of the quotient morphisms $\nu_I:\ X_I \rightarrow Y_I$ and $\nu_{I\!I}:\ X_{I\!I} \rightarrow Y_{I\!I}$ are isomorphic to $G$.\\
If $m<n$ and $m$ is even, then the general fibers of $\nu_I$ and $\nu_{I\!I}$ are isomorphic to $G/G'$, where $G' \subset G$ is the subgroup defined by (\ref{defHsymp}). 
\end{proposition}

\begin{corollaire} \label{fctHSp2} 
The Hilbert function $h_0$ defined above is given by:  
$$\forall M \in \Irr(G),\ h_0(M)=\left\{
    \begin{array}{ll}
       \dim(M)  &\text{ if } m \geq n;  \\
       \dim(M^{G'})    &\text{ if } m <n \text{ and $m$ is even;} 
    \end{array}
\right.
$$
where $G' \subset G$ is the subgroup defined by (\ref{defHsymp}). 
\end{corollaire}

\subsection{The reduction principle for the main component}  \label{sssympSpn}
In this section, we give the guidelines to prove the reduction principle when $G=Sp(V)$ (Proposition \ref{reduction37}). The strategy is the same as for $GL(V)$ (see Section \ref{MropRRED}), but as the symplectic reduction $\mu^{-1}(0)/\!/G$ is reducible when $m \leq n$ and $m$ is even, it seems necessary to give some additional details.

As $\mu^{-1}(0)$ is a $G \times H$-stable subvariety of $W$, it follows from \cite[Lemma 3.3]{Br} that the invariant Hilbert scheme
\begin{equation*}
\HH:=\Hilb_{h_0}^{G}(\mu^{-1}(0)) 
\end{equation*}
is a $H$-stable closed subscheme of $\Hilb_{h_0}^G(W)$. 
As we aim at constructing desingularizations of the irreducible components of $\mu^{-1}(0)/\!/G$, we consider the two $H$-stable closed subschemes $\HHx:=\Hilb_{h_0}^{G}(X_I)$ and $\HHy:=\Hilb_{h_0}^{G}(X_{I\!I})$ instead of $\HH$ when $m \leq n$ and $m$ is even. Let us note that if we fix $y_0 \in O(E) \backslash SO(E)$ and make $H$ act on $X_{I\!I}$ by $(y_0y{y_0}^{-1}).x$ for every $y \in H$ and every $x \in X_{I\!I}$, then $\phi:\ X_I \rightarrow X_{I\!I},\ x \mapsto y_0.x$ is a $G \times H$-equivariant isomorphism, and thus $\HHx \cong \HHy$ as a $H$-scheme. We denote by $\HHxp$ resp. by $\HHyp$, the main component of $\HHx$ resp. of $\HHy$. We always have the (set-theoretic) inclusion $\HHx \cup \HHy \subset \HH$, but this may not be an equality. 
If $m>n$, then $\mu^{-1}(0)/\!/G$ is irreducible, and we denote by $\HHp$ the main component of $\HH$.

The scheme $\Hilb_{h_0}^{G}(W)$ was studied in \cite{Terp}. In particular, we obtained

\begin{proposition} \label{moorppgr2}  
\emph{(\cite[\S 1.5.1]{Terp})} Let $h_0$ be the Hilbert function given by Corollary \ref{fctHSp2}, and let $H=SO(E)$ acting naturally on $\Gr(2m-h_0(V^*),E)$.  
Then there exists a $H$-equivariant morphism
$$\rho  : \ \Hilb_{h_0}^G(W)  \rightarrow  \Gr(2m-h_0(V^*),E)$$
given on closed points by $[Z] \mapsto \Ker(f_Z)$, where $f_Z:\ E \cong \Mor^G(W,V^*) \rightarrow \Mor^G(Z,V^*)$ is the restriction map. 
\end{proposition}

We identify $\Gr(2m-h_0(V^*),E)$ with $\Gr(h_0(V^*),E^*)$. By Corollary \ref{fctHSp2}, if either $m>n$ or $m \leq n$, $m$ even, then $h_0(V^*)=N:=\min(m,n)$. The non-degenerate quadratic form $q$ on $E$ gives a canonical isomorphism $E \cong E^*$. In particular, $q$ identifies with a non-degenerate quadratic form on $E^*$. For $i=0, \ldots,N$, we denote   
$$A_i:=\{L \in \Gr(N,E^*)\ |\ {q}_{|L} \text{ is of rank } i\}.$$
If $m>n$, then the $A_i$ are the $n+1$ orbits for the action of $H$ on $\Gr(n,E^*)$. \\
However, if $m \leq n$, then the $A_i$ are $H$-orbits for $i=1,\ldots,m$, but the isotropic Grassmannian $A_0=\OG(m,E^*)$ is the union of two $H$-orbits, denoted by $\OG^I$ and by $\OG^{I\!I}$, which are exchanged by the action of any element of $O(E) \backslash SO(E)$.\\
In any case, we have     
$$ \OG(N,E^*)=\overline{A_0} \subset \overline{A_1} \subset \cdots \subset \overline{A_N}=\Gr(N,E^*).$$
Let us now fix some notation: 
\begin{itemize} \renewcommand{\labelitemi}{$\bullet$}
\item $L_0 \in A_0$, and $P$ the parabolic subgroup of $H$ stabilizing $L_0$;
\item $W':=\Hom(E/L_0^{\perp},V)$, which identifies with a $G \times P$-module contained in $\mu^{-1}(0)$; and
\item $\HH':=\Hilb_{h_0}^{G}(W')$, and $\HH'^{\mathrm{main}}$ its main component.
\end{itemize}

It must be emphasized that, if either $m> n$ or $m\leq n$, $m$ even, then the Hilbert function of the general fibers of the quotient morphism  $W' \rightarrow W'/\!/G$ coincides with the Hilbert function $h_0$ of Corollary \ref{fctHSp2} (in particular, $\HH'^{\mathrm{main}}$ is well defined).

Proceeding as for Lemma \ref{versX0}, one may check that, if $m>n$ resp. if $m \leq n$ with $m$ even, then the morphism $\rho$ of Proposition \ref{moorppgr2} sends $\HHp$ resp. $\HHxp$ and $\HHyp$, onto $A_0$. More precisely, if $m \leq n$ and $m$ is even, then $\rho$ sends $\HHxp$ onto one of the irreducible component of $A_0$, and $\HHyp$ onto the other component. Up to the exchange of these two components, we can suppose that $\rho$ sends $\HHxp$ onto $\OG^I$, and $\HHyp$ onto $\OG^{I\!I}$.\\ 
It follows that the restriction of $\rho$ equips $\HHp$ resp. $\HHxp$, resp. $\HHyp$, with a structure of a $H$-homogeneous fiber bundle over $A_0$ resp. over $\OG^I$, resp. over $\OG^{I\!I}$. Hence, it is enough to determine the fiber $F_0$ over $L_0$ to determine $\HHp$ resp. $\HHxp$, resp. $\HHyp$. Proceeding as in Section \ref{MropRRED}, we obtain that $F_0$ is isomorphic to $\HH'^{\mathrm{main}}$ as a $P$-scheme. We deduce       

\begin{proposition} \label{reduction37}
With the above notation, we have the following $H$-equivariant isomorphisms:
\begin{itemize}
\item If $m>n$, then 
$$\HHp \cong H {\times}^{P} \HH'^{\mathrm{main}}.$$
\item If $m \leq n$, $m$ even, and $L_0 \in \OG^I$ resp. $L_0 \in \OG^{I\!I}$, then 
$$\HHxp \cong H {\times}^{P} \HH'^{\mathrm{main}} \text{ resp. $\HHyp \cong H {\times}^{P} \HH'^{\mathrm{main}}$.}$$
\end{itemize}
\end{proposition}

\subsection{Proofs of Theorems A and B for $Sp(V)$}  \label{pprrthe}

In this section, we proceed as in Section \ref{proofGLn} to prove Theorems A and B when $G=Sp(V)$. Before going any further, let us mention that the case $n=2$, $m=3$ was already handled by Becker in \cite{Tanja2}. In this situation, $\mu^{-1}(0)/\!/G$ is a closure of a nilpotent orbit that admits two Springer desingularizations, and Becker showed that $\gamma:\ \Hilb_{h_0}^{G}(\mu^{-1}(0)) \rightarrow \mu^{-1}(0)/\!/G$ is a desingularization that dominates them both. To obtain this result, Becker first used the existence of natural morphisms from the invariant Hilbert scheme to Grassmannians to identify $\Hilb_{h_0}^{G}(\mu^{-1}(0))^{\mathrm{main}}$ with the total space of a homogeneous line bundle over a Grassmannian, and then she showed that $\Hilb_{h_0}^{G}(\mu^{-1}(0))=\Hilb_{h_0}^{G}(\mu^{-1}(0))^{\mathrm{main}}$ by computing the tangent space of $\Hilb_{h_0}^{G}(\mu^{-1}(0))$ at every point of the main component. 

Let us now recall the following result:

\begin{theoreme} \emph{(\cite[\S 1, Theorem]{Terp1})} \label{Terpy2}
Let $G=Sp(V)$, let $W=\Hom(E,V)$, and let $h_W$ be the Hilbert function of the general fibers of the quotient morphism $W \rightarrow W/\!/G$. 
We denote $n:=\dim(V)$, $e:=\dim(E)$, and we denote by $Y_0$ the blow-up of $W/\!/G=\Lambda^2(E^*)^{\leq n}:=\{Q \in \Lambda^2(E^*)\ |\ \rk(Q) \leq n\}$ at $0$. 
In the following cases, the invariant Hilbert scheme $\HH':=\Hilb_{h_W}^{G}(W)$ is a smooth variety, and the Hilbert-Chow morphism is the succession of blow-up described as follows:
\begin{itemize}
\item if $n \geq 2 e-2$, then $\HH'\cong W/\!/G=\Lambda^2(E^*)$;
\item if $e>n=2$ or $e=n=4$, then $\HH' \cong Y_0$;
\item if $e>n=4$, then $\HH'$ is isomorphic to the blow-up of $Y_0$ along the strict transform of $\Lambda^2(E^*)^{\leq 2}$.
\end{itemize}
\end{theoreme}

If $m>n$, then we denote by $T$ the tautological bundle over $A_0=\OG(n,E^*)$. If $m \leq n$ and $m$ is even, then we denote by $T_{I}$ resp. by $T_{I\!I}$, the tautological bundle over $\OG^I$ resp. over $\OG^{I\!I}$. We deduce from Proposition \ref{reduction37} and Theorem \ref{Terpy2} the following $H$-equivariant isomorphisms
\begin{equation*}  
\HHp \cong \left\{
    \begin{array}{ll}
       \Lambda^2(T)         &\text{ if } m>n=2;\\
        Bl_0(\Lambda^2(T))  &\text{ if } m>n=4;
    \end{array}
\right.
\end{equation*}
\begin{equation}  \label{eqm2}
\HH_{\bullet}^{\mathrm{main}} \cong \left\{
    \begin{array}{ll}
       \Lambda^2(T_{\bullet})         &\text{ if } n \geq 2m-2 \text{ and $m$ is even};\\
        Bl_0(\Lambda^2(T_{\bullet}))  &\text{ if } m=n=4;
    \end{array}
\right.
\end{equation}
where ${\bullet}$ stands for $I$ or $I\!I$, and $Bl_0(.)$ denotes the blow-up along the zero section. In all these cases, the main component of the invariant Hilbert scheme is smooth, and thus the Hilbert-Chow morphism $\gamma:\ \HHp \rightarrow \mu^{-1}(0)/\!/G$ resp. $\gamma:\ \HH_{\bullet}^{\mathrm{main}} \rightarrow Y_{\bullet}$, is a desingularization.

It remains to compare $\gamma$ with the Springer desingularizations (when they exist) of the irreducible components of $\mu^{-1}(0)/\!/G$. We saw in Section \ref{gen} that the irreducible components of $\mu^{-1}(0)/\!/G$ have Springer desingularizations if and only if $m \leq n+1$. We then distinguish between the following cases:

\begin{enumerate}
\item If $m\leq n+1$ and $m$ is odd, then $\mu^{-1}(0)/\!/G$ admits two Springer desingularizations, which are given by the cotangent bundles $\TT_{I}^{*}$ and $\TT_{I\!I}^{*}$ over $\OG^{I}$ and $\OG^{I\!I}$ respectively. The natural action of the orthogonal group $O(E)$ on $\OG(m,E^*)$ induces an action on the cotangent bundle $\TT^* \OG(m,E^*)$ that exchanges  $\TT_{I}^{*}$ and $\TT_{I\!I}^{*}$. On the other hand, it follows from Remark \ref{veryuseful}, that the group $O(E)$ stabilizes $\HHp$, and thus $\gamma:\ \HHp \rightarrow \mu^{-1}(0)/\!/G$ cannot be a Springer desingularization. \\
However, if $n \in \{2,4\}$ and $m=n+1$, then one may prove that $\gamma$ dominates the two Springer desingularizations of $\mu^{-1}(0)/\!/G$ (see \cite[Introduction]{Tanja2} for the case $n=2$, the case $n=4$ being analogous).

\item If $m \leq n$ and $m$ is even, then $Y_{\bullet}$ has a unique Springer desingularization, which is given by the cotangent bundle $\TT_{\bullet}^{*} \cong \Lambda^2(T_{\bullet})$ over $\OG^{\bullet}$. Proceeding as we did for $GL(V)$ in Section \ref{proofGLn}, one may prove that $\gamma:\ \HH_{\bullet}^{\mathrm{main}} \rightarrow Y_{\bullet}$ is the Springer desingularization if and only if $n \geq 2m-2$.\\
In addition, if $m=n=4$, then by (\ref{eqm2}) we have $\HH_{\bullet}^{\mathrm{main}} \cong Bl_0(\TT_{\bullet}^{*})$, and thus $\gamma$ dominates the unique Springer desingularization of $Y_{\bullet}$.  
\end{enumerate}

\subsection{Study of the case \texorpdfstring{$n=2$}{n=2}} \label{cassympn2}

In this section, we suppose that $m \geq n=2$ (the case $m=1$ being trivial). We will prove that if $m \geq 3$ resp. if $m=2$, then $\HH$ resp. $\HH_{\bullet}$ (where $\bullet$ stands for $I$ or $I\!I$), is irreducible. In particular, the geometric properties of the invariant Hilbert scheme for $G=Sp(V)$ are quite different from the case of $G=GL(V)$ studied in Section \ref{posisimpy}. Let us recall that the case $m=3$, $n=2$ was treated by Becker in \cite{Tanja2}; she showed that $\HH$ is the total space of a line bundle over $\OG(2,E^*)$.

We have $G \cong Sp_2=SL_2$, and the morphism of Proposition \ref{moorppgr2} is $\rho:\ \Hilb_{h_0}^{G}(W) \rightarrow \Gr(2,E^*)$. Denoting $\hh^{\leq 2}:=\{ f \in \hh \ | \ \rk (f) \leq  2 \}$, there is a $H$-equivariant isomorphism 
\begin{equation}  \label{isomonn}
\PP(\hh^{\leq 2}) \cong \Gr(2,E^*),  
\end{equation}
and thus we can consider the morphism $\rho':\ \Hilb_{h_0}^{G}(W) \rightarrow \PP(\hh^{\leq 2})$ induced by $\rho$.
By Proposition \ref{descqu}, we have
$$\mu^{-1}(0)/\!/G= \left\{
    \begin{array}{ll}
           \overline{\OO_{[2^2,1^{2m-4}]}} & \text{ if } m \geq 3;\\
      \overline{\OO_{[2^2]}^{I}} \cup \overline{\OO_{[2^2]}^{I\!I}} & \text{ if } m=2.
     \end{array}
\right.$$

\begin{proposition}  \label{wxcv}
We equip all the invariant Hilbert schemes with their reduced structures. If $m>n=2$, then $\HH=\HHp$ is a smooth variety isomorphic to
$$Bl_0(\overline{\OO_{[2^2,1^{2m-4}]}}):=\left \{(f,L) \in \overline{\OO_{[2^2,1^{2m-4}]}} \times \PP (\overline{\OO_{[2^2,1^{2m-4}]}}) \ \mid  \ f \in L  \right \},$$ 
and the Hilbert-Chow morphism $\gamma:\ \HH \rightarrow \mu^{-1}(0)/\!/G$ is the blow-up of $\overline{\OO_{[2^2,1^{2m-4}]}}$ at $0$.
If $m=n=2$, then $\Hilb_{h_0}^{G}(\mu^{-1}(0))=\HHx \cup \HHy$ is the union of two smooth irreducible components isomorphic to $Bl_0(\overline{\OO_{[2^2]}^{I}})$ and $Bl_0(\overline{\OO_{[2^2]}^{I\!I}})$ respectively, and the set-theoretic intersection $\HHx \cap \HHy$ is formed by the homogeneous ideals of $\CC[\mu^{-1}(0)]$. Moreover, the Hilbert-Chow morphism $\gamma:\ \HHx \rightarrow \overline{\OO_{[2^2]}^{I}}$ resp. $\gamma:\ \HHy \rightarrow \overline{\OO_{[2^2]}^{I\!I}}$, is the blow-up of $\overline{\OO_{[2^2]}^{I}}$ resp. of $\overline{\OO_{[2^2]}^{I\!I}}$, at $0$. 
\end{proposition}

\begin{proof}
The proofs for the cases $m=2$ and $m \geq 3$ are quite similar, and thus we will only consider the case $m \geq 3$ (which is simpler in terms of notation!). 
Using arguments similar to those used to prove Proposition \ref{casSympn1}, we obtain a closed embedding
$$\gamma \times \rho':\ \HH \hookrightarrow \YY:=\left \{(f,L) \in \overline{\OO_{[2^2,1^{2m-4}]}} \times \PP ({\hh}^{\leq 2}) \ \mid  \ f \in L  \right \}.$$ 
One may check that $\YY$ is the union of the two irreducible components $C_1$ and $C_2$ defined by: 
\begin{align*}
\bullet \ C_1&:=Bl_0(\overline{\OO_{[2^2,1^{2m-4}]}}); and \\
\bullet \ C_2&:=\left \{(0,L) \in \overline{\OO_{[2^2,1^{2m-4}]}} \times \PP ({\hh}^{\leq 2}) \right \} \cong \PP ({\hh}^{\leq 2}).
\end{align*}
The components $C_1$ and $C_2$ are of dimension $4m-6$ and $4m-4$ respectively. The morphism $\gamma \times \rho'$ sends $\HHp$ into $C_1$; the varieties $\HHp$ and $C_1$ have the same dimension, hence $\gamma \times \rho':\ \HHp \rightarrow C_1$ is an isomorphism.\\ 
Now it follows from \cite[Proposition 3.3.13]{Terp} that the component $C_2$ identifies with the closed subset of $\Hilb_{h_0}^{G}(W)$ formed by the homogeneous ideals of $\CC[W]$. Let us describe this identification. If $L \in C_2 \cong \PP ({\hh}^{\leq 2})$, then we denote by $I_L$ the ideal of $\CC[W]$ generated by the homogeneous $G$-invariants of positive degree of $\CC[W]$, and by the $G$-module $L^{\perp} \otimes V \subset \CC[W]_1 \cong E \otimes V$, where $L$ is identified with a 2-dimensional subspace of $E^*$ via the isomorphism (\ref{isomonn}). Let us show that $I_L$ is a point of $\HH$ if and only if $L \in \OG(2,E^*)$; the result will follow since $\PP (\overline{\OO_{[2^2,1^{2m-4}]}})$ identifies with $\OG(2,E^*)$ via the isomorphism (\ref{isomonn}), and since $\left \{(0,L) \in \overline{\OO_{[2^2,1^{2m-4}]}} \times \PP (\overline{\OO_{[2^2,1^{2m-4}]}})  \right\}$ is a subvariety of $C_1$.\\     
We denote $W':=\Hom(E/L^{\perp},V)$, then
$$\CC[W']_2 \cong (S^2(E/L^{\perp}) \otimes S^2(V)) \oplus (\Lambda^2(E/L^{\perp}) \otimes \Lambda^2(V))$$ 
as a $G$-module. Let $I'_L$ be the ideal of $\CC[W']$ generated by $\Lambda^2(E/L^{\perp}) \otimes \Lambda^2(V) \subset \CC[W']_2$, then one may check (using \cite[Proposition 3.3.13]{Terp}) that 
$$\CC[W]/I_L \cong \CC[W']/I'_L \cong \bigoplus_{M \in \Irr(G)}  M^{\oplus \dim(M)}$$ 
as a $G$-module. Hence
\begin{align*}
I_L \in \HH & \Leftrightarrow I_L \cap \CC[W]_2 \supset E_0 \otimes S^2(V), \text{ where $E_0$ is the trivial representation of $H$;} \\
             &  \Leftrightarrow  {q}_{|L}=0, \text{ where $q$ is the quadratic form preserved by $H$;} \\
             & \Leftrightarrow L \in \OG(2,E^*).
\end{align*} 
\end{proof}

\begin{remarque}
In the proof of Proposition \ref{wxcv}, we showed that if $m > n=2$, then the homogeneous ideals of $\HH$ are contained in $\HHp$. Using analogous arguments, one may check that this statement is true more generally when  $m > n \geq 2$. \\
\end{remarque}

\noindent \textbf{Acknowledgments:} I am deeply thankful to Michel Brion for proposing this subject to me, for a lot of helpful discussions, and for his patience. I thank Tanja Becker for exchange of knowledge on invariant Hilbert schemes by e-mail and during her stay in Grenoble in October 2010. I also thank Bart Van Steirteghem for helpful discussions during his stay in Grenoble in Summer 2011. \\

\noindent \texttt{Université Grenoble I, Institut Fourier,}\\
\texttt{UMR 5582 CNRS-UJF, BP 74,}\\
\texttt{38402 St. Martin d'Hères Cédex, FRANCE}\\
\textit{E-mail address:} \texttt{ronan.terpereau@ujf-grenoble.fr}

\end{document}